\documentclass{amsart}
\usepackage{amssymb}
\usepackage{amscd}
\usepackage{multicol}
\usepackage{MnSymbol}
\usepackage{tikz}
\usepackage[OT2,T1]{fontenc}
\usepackage{centernot}
\usepackage{mathrsfs}
\DeclareSymbolFont{cyrletters}{OT2}{wncyr}{m}{n}
\DeclareMathSymbol{\Sha}{\mathalpha}{cyrletters}{"58}

\title[On multi-poly-Bernoulli-Carlitz numbers]
{On multi-poly-Bernoulli-Carlitz numbers
}
\author{Ryotaro Harada}
\address{Graduate School of Mathematics, Nagoya University, 
Furo-cho, Chikusa-ku, Nagoya 464-8602 Japan }
\email{m15039r@math.nagoya-u.ac.jp}
\date{March 27, 2018}

\newtheorem{thm}{Theorem}
\newtheorem{lem}[thm]{Lemma}

\newtheorem{prop}[thm]{Proposition}  
\theoremstyle{remark}
        
\theoremstyle{definition}
\newtheorem{defn}[thm]{Definition}
\newtheorem{rem}[thm]{Remark}
\newtheorem{nota}[thm]{Notation}     

\newtheorem{eg}[thm]{Examples}

\begin{document}
\bibliographystyle{amsalpha+}

\begin{abstract}
We introduce multi-poly-Bernoulli-Carlitz numbers, function field analogues of multi-poly-Bernoulli numbers of Imatomi-Kaneko-Takeda. 
We explicitly describe multi-poly-Bernoulli Carlitz numbers in terms of the Carlitz factorial and the Stirling-Carlitz numbers of the second kind and also show their relationships with function field analogues of finite multiple zeta values.
%
\end{abstract}

\maketitle
\tableofcontents
\setcounter{section}{-1}
\section{Introduction}
In this paper, we introduce and study function field analogues of the Bernoulli numbers.  

In 1997, M. Kaneko introduced and investigated generalizations of the Bernoulli numbers, poly-Bernoulli numbers in \cite{MK}. He obtained explicit formulae for poly-Bernoulli numbers which includes the second kind of Stirling numbers. Moreover, He and T. Arakawa found that they are also related to the Arakawa-Kaneko zeta functions at non-positive integers in \cite{AK}. From 2000, several multi-poly-Bernoulli numbers, generalizations of poly-Bernoulli numbers, were posted by Hamahata-Masubuchi \cite{HM}, Imatomi-Kaneko-Takeda \cite{IKT} and M.-S. Kim-T. Kim \cite{KiK} in different ways each other. In \cite{IKT}, K. Imatomi, M. Kaneko, E. Takeda established relationships between multi-poly-Bernoulli numbers and finite multiple zeta values by obtaining some fundamental formulae. 

In 1935, L. Carlitz \cite{Ca} introduced and investigated function field analogues of the Bernoulli numbers, the Bernoulli-Carlitz numbers $BC_n$. By using them, he obtained an analogue of Euler's famous formula $\zeta(m)=-\frac{(2\pi i)^m}{2(m!)}B_m$ (for even $m$) in \cite{Ca} and the von Staudt-Clausen theorem in \cite{Ca37, Ca40}. The latter result was revisited and put in a more conceptual setting by D. Goss in \cite{Go}.\footnote{ An analogue of von Staudt-Clausen theorem stated in \cite{Ca37, Ca40, Go} was corrected by L. Carlitz \cite{Ca41} for $q=2$.} 
E. Gekeler proved several identities for the Bernoulli-Carlitz numbers in \cite{Gek}. Furthermore, H. Kaneko and T. Komatsu obtained explicit formulae for them by using function field analogues of the Stirling numbers in \cite{KK}. 
 In this paper, we introduce in \S \ref{2chp} multi-poly-Bernoulli-Carlitz numbers as function field analogues of multi-poly-Bernoulli numbers and also discuss generalizations of the vanishing condition $BC_n=0\ \ ( q-1 \nmid n)$ shown in \cite{Ca37} and explicit formulae $BC_n=\sum_{j=0}^{\infty}\frac{(-1)^jD_j}{L^2_{j}}{n\brace q^j-1}_{C}$ shown in \cite{KK}.  
In \S \ref{mainchp} we show that multi-poly-Bernoulli-Carlitz numbers with special indices are expressed by Bernoulli-Carlitz numbers. We also show their connection to finite multiple zeta values in function field which were introduced by C.-Y. Chang and Y. Mishiba \cite{CM} as finite variants of Thakur's multiple zeta values in \cite{T1}. 
\section{Notations and Definitions}
\label{No}
\subsection{Notations}
We recall the following notation. 
\begin{itemize}
\setlength{\leftskip}{1.0cm}
\item[$q$] \quad  a power of a prime number $p$.  
\item[$\mathbb{F}_q$] \quad a finite field with $q$ elements.
\item[$\theta$, $t$] \quad independent variables.
\item[$A$] \quad the polynomial ring $\mathbb{F}_q[\theta]$.
\item[$A_{+}$] \quad the set of monic polynomials in $A$.   
\item[$k$] \quad the rational function field $\mathbb{F}_{q}(\theta)$.
\item[$k_{\infty}$]\quad $\mathbb{F}_{q}((1/\theta))$, the completion of $k$ at $\infty$.
\item[$D_i$]\quad $\prod^{i-1}_{j=0}(\theta^{q^i}-\theta^{q^j})\in A_{+}$ where $D_0:=1$.
\item[$L_i$]\quad $\prod^{i}_{j=1}(\theta-\theta^{q^j})\in A_{+}$ where $L_0:=1$.
\item[$\Gamma_{n+1}$]\quad the Carlitz gamma, $\prod_{i}D_i^{n_i}$ ($n = \sum_{i}n_iq^i\in\mathbb{Z}_{\geq0} \ (0\leq n_i\leq q-1)$).  
\item[$\Pi(n)$]\quad the Carlitz factorial, $\Gamma_{n+1}$
\end{itemize}
\subsection{Definition of finite multiple zeta values}
In this subsection, we recall the definition of finite multiple zeta values and its function field analogues which were introduced in \cite{CM}. 

\subsubsection{Characteristic 0 case}
We begin this subsection by recalling the finite multiple zeta values those were introduced by M. Kaneko and D. Zagier in \cite{KZ}.
\begin{defn}[\cite{KZ}]
We set a $\mathbb{Q}$-algebra as follows:
\[
\mathscr{A}:=\prod_{p}\mathbb{Z}\big/p\mathbb{Z}\bigg/\bigoplus_{p}\mathbb{Z}\big/p\mathbb{Z}
\] 
where $p$ runs over all prime numbers. For $\mathfrak{s}:=(s_1, \ldots, s_r)\in\mathbb{Z}^r$, the {\it finite multiple zeta values} are defined as follows:
\[
	\zeta_{\mathscr A}(\mathfrak{s}):=(\zeta_{\mathscr A}(\mathfrak{s})_{(p)})\in{\mathscr A}
\]
where
\[
	\zeta_{\mathscr A}(\mathfrak{s})_{(p)}:=\sum_{p>m_1>\cdots>m_r>0}\frac{1}{m^{s_1}_1\cdots m^{s_r}_r}\in\mathbb{Z}\big/ p\mathbb{Z}.
\]
\end{defn}

\subsubsection{Characteristic $p$ case}

Next, let us turn into function field situation.
In 1935, L. Carlitz \cite{Ca} considered an analogue of the {\it Riemann zeta values} in function field which we call the {\it Carlitz zeta values}.
For $s \in \mathbb{N}$, they are defined by
\[
  \zeta_{A}(s):=\sum_{a\in A_{+}}\frac{1}{a^s}\in k_{\infty}.
\]
D. S. Thakur \cite{T1} generalized this definition to that of {\it multiple zeta values in} $\mathbb{F}_q[t]$, which are defined for ${\mathfrak s}=(s_1, \ldots, s_r)\in\mathbb{N}^r$, 
\begin{align*}
  \zeta_{A}({\mathfrak s})&:=\sum_{\substack{\deg a_1>\cdots>\deg a_r\geq 0\\a_1, \ldots, a_r\in A_+}}\frac{1}{a^{s_1}_1\cdots a^{s_r}_r}\in k_{\infty}.
\end{align*}
Also, Chang-Mishiba and D. S. Thakur concerned $v$-adic variant (\cite{CM1}, \cite{T1}) and finite variant (\cite{CM}, \cite{T2}).   
In this paper, we consider Chang and Mishiba's finite variant (\cite{CM}).
\begin{defn}[\cite{CM}, (2.1)]
We set a $k$-algebra as follows:
\[
\mathscr{A}_{k}:=\prod_{\wp} A\big/ \wp A\bigg/ \bigoplus_{\wp} A\big/ \wp A
\]
where $\wp$ runs over all monic irreducible polynomials in $A$. For $\mathfrak{s}:=(s_1, \ldots, s_r)\in\mathbb{N}^r$ and a monic irreducible polynomial $\wp\in A$, {\it finite multiple zeta values} are defined as follows:
\[
	 \zeta_{\mathscr{A}_{k}}(\mathfrak{s}):=(\zeta_{\mathscr{A}_{k}}(\mathfrak{s})_{\wp})\in\mathscr{A}_{k}
\]
where 
\[
	\zeta_{\mathscr{A}_{k}}(\mathfrak{s})_{\wp}:=\sum_{\substack{\deg\wp>\deg a_1>\cdots>\deg a_r\geq 0\\a_1, \ldots, a_r\in A_+}}\frac{1}{a^{s_1}_1\cdots a^{s_r}_r}\in A\big/\wp A.
\]
\end{defn}

\subsection{Definition of (finite Carlitz) multiple polylogarithms}
\label{fcmpl}

In this subsection, we recall the definition of multiple polylogarithms in characteristic 0 and $p$.
\subsubsection{Characteristic 0 case}
\begin{defn}
 For $\mathfrak{s}:=(s_1, \ldots, s_r)\in\mathbb{Z}^r$, the {\it multiple polylogarithm series}
 ${\rm Li}_{\mathfrak s}(z_1, \ldots, z_r)$ are defined as the following series of $r$-variables $z_1, \ldots, z_r$:
\[
	{\rm Li}_{\mathfrak s}(z_1, \ldots, z_r):=\sum_{m_1>\cdots>m_r>0}\frac{z_1^{m_1}\cdots z_r^{m_r}}{m_1^{s_1}\cdots m_r^{s_r}}\in\mathbb{Q}[[z_1, \ldots, z_r]].
\]
\end{defn}

\subsubsection{Characteristic $p$ case}
In 2014, C.-Y. Chang \cite{C1} introduced the Carlitz multiple polylogarithms as function field analogues of the multiple polylogarithms. 
\begin{defn}[\cite{C1}, Definition 5.1.1]
\label{cmpl}
For ${\mathfrak s}=(s_1, \ldots, s_r)\in\mathbb{N}^r$, the {\it Carlitz multiple polylogarithms} are defined as the following series of $r$-variables $z_1, \ldots, z_r$:
\[
	Li_{\mathfrak{s}}(z_1, \ldots, z_r):=\sum_{i_1>\cdots>i_r\geq0}\frac{z_1^{q^{i_1}}\cdots z_r^{q^{i_r}} }{L_{i_1}^{s_1}\cdots L_{i_r}^{s_r}}\in k[[z_1, \ldots, z_r]].
\]
\end{defn}

\begin{rem}
We recover the {\it Carlitz logarithms} in the case of $r=1$ and $s_1=1$
\[
\log_{C}(z):=\sum_{i\geq0}\frac{z^{q^i} }{L_i}\in k[[z]].
\]
\end{rem}

In \cite{CM}, C.-Y. Chang and Y. Mishiba introduced finite Carlitz multiple polylogarithms, a finite variant of the Carlitz multiple polylogarithms.
\begin{defn}[\cite{CM}, (3.1)]
\label{fcmpldef}
For $\mathfrak{s}=(s_1,\ldots, s_r)\in\mathbb{N}^r$ and $r$-tuple of variables $\mathfrak{z}=(z_1, \ldots, z_r)$, {\it finite Carlitz multiple polylogarithms} are defined as follows:
\[
	Li_{\mathscr{A}_{k}, \mathfrak{s}}(\mathfrak{z}):=(Li_{\mathscr{A}_{k}, \mathfrak{s}}(z_1, \ldots, z_r)_{\wp})\in\mathscr{A}_{k, \mathfrak{z}}
\]
where 
\[
	Li_{\mathscr{A}_{k}, \mathfrak{s}}(z_1, \ldots, z_r)_{\wp}:=\sum_{\deg\wp>i_1>\cdots>i_r\geq0}\frac{z^{q^{i_1}}_1\cdots z^{q^{i_r}}_r}{L^{s_1}_{i_1}\cdots L^{s_r}_{i_r}}\ \text{mod}\ \wp \in A[z_1, \ldots, z_r]/\wp A. 
\]
Here $\mathscr{A}_{k, \mathfrak{z}}$ is the following quotient ring
 \[
 	\mathscr{A}_{k, \mathfrak{z}}:=\prod_{\wp}A[\mathfrak{z}]/\wp A[\mathfrak{z}]\bigg/ \bigoplus_{\wp}A[\mathfrak{z}]/\wp A[\mathfrak{z}]
 \]
(we put $A[\mathfrak{z}]:=A[z_1, \ldots, z_r]$).
\end{defn}

In \cite{CM}, they established an explicit formula expressing $\zeta_{\mathscr{A}_{k}}(\mathfrak{s})$ as a $k$-linear combination of $Li_{\mathscr{A}_{k}, {\mathfrak{s}}}(z_1, \ldots, z_r)_{\wp}$ evaluated at some integral points. Before we recall it, let us prepare the Anderson-Thakur polynomial.



\begin{defn}[\cite{AT}, (3.7.1)]
Let $\theta, t, x$ be independent variables. For $n\in \mathbb{Z}_{\geq0}$, {\it Anderson-Thakur polynomial} $H_n\in A[t]$ is defined by 
\[
	\biggl\{1-\sum^{\infty}_{i=0}\frac{\prod^{i}_{j=1}\bigl( t^{q^i}-\theta^{q^j} \bigr)}{D_i|_{\theta=t}}x^{q^i} \biggr\}^{-1}=\sum^{\infty}_{n=0}\frac{H_n}{\Gamma_{n+1}|_{\theta=t}}x^n.
\]
\end{defn}
\begin{rem}\label{chen}
We note that $H_n=1$ for $0\leq n\leq q-1$.
\end{rem}
G. W. Anderson and D. S. Thakur obtained the following series expansion for $H_{s_i-1}$. 
\begin{prop}[\cite{AT}, (3.7.3), (3.7.4) and \cite{AT09}, 2.4.1]
We consider $r$-tuple $\mathfrak{s}=(s_1, \ldots, s_r)\in\mathbb{N}^r$. For each $s_i\in\mathbb{N}$, the Anderson-Thakur polynomial $H_{s_i-1}$ is expanded as follows.
\begin{align}\label{atpoly}
	H_{s_i-1}=\sum^{m_i}_{j=0}u_{ij}t^j
\end{align}
where $u_{ij}\in A$ satisfying 
\[
	|u_{ij}|_{\infty}<q^{\frac{s_iq}{q-1}}\ \text{and}\ u_{im_i}\neq0.
\]
\end{prop}
Here we note that $|\ \ |_{\infty}$ is the non-Archimedian absolute value on $k$ so that $|\theta|_{\infty}=q$.
\begin{nota}\label{notasym}
We set following symbols which are introduced in \cite{CM}:
\[
	J_{\mathfrak s}:=\{0, 1, \ldots, m_1\}\times\cdots\times\{0, 1, \ldots, m_r\}.
\]  
For each ${\bf j}=(j_1, \ldots, j_r)\in J_{\mathfrak s}$, we set
\[
	{\bf u_j}:=(u_{ij_i}, \ldots, u_{rj_r})\in A^r,
\]
and
\[
	a_{\bf j}:=a_{\bf j}(t):=t^{j_1+\cdots+j_r}.
\]
Here $u_{ij}$ are the coefficients of \eqref{atpoly}.
\end{nota}

\begin{eg}
We note that when ${\mathfrak s}=(s_1, \ldots, s_r)=(1, \ldots, 1)$, by Remark \ref{chen} and Proposition \ref{atpoly}, we have $J_{\mathfrak s}=\{ 0\}\times\cdots\times\{0\}$ and ${\bf u_j}=(1, \ldots, 1)$ for ${\bf j}\in J_{\mathfrak s}$.  
\end{eg}

The following equation was obtained by C.-Y. Chang and Y. Mishiba in \cite{CM}.
\begin{prop}[\cite{CM}, p.1056]
\label{fmzvfmpl}
For ${\mathfrak s}=(s_1, \ldots, s_r)\in\mathbb{N}^r$, let $\wp\in A$ be a monic irreducible polynomial which satisfy $\wp \nmid \Gamma_{s_1}\cdots\Gamma_{s_r}$. Then we have

\[
	\zeta_{\mathscr{A}_{k}}({\mathfrak s})_{\wp}=\frac{1}{\Gamma_{s_1}\cdots\Gamma_{s_r}}\sum_{{\bf j}\in J_{\mathfrak s}}a_{\bf j}(\theta)Li_{{\mathscr A}_k, {\mathfrak s}}(\bf u_j)_{\wp}.
\]
\end{prop}

\section{Multi-poly-Bernoulli(-Carlitz) numbers}
In this section, we define multi-poly-Bernoulli-Carlitz numbers which are function field analogues of multi-poly-Bernnoulli numbers.

\subsection{Characteristic 0 case}\label{2ch0}
The {\it Bernoulli numbers} $B_n\ (n=0, 1, \ldots)$ are rational numbers defined by the following generating function
\begin{align}\label{defb}
	\sum_{n=0}^{\infty}B_n\frac{z^n}{n!}:=\frac{ze^z}{e^z-1}.
\end{align}
It is known that we have the following equation
\begin{align*}
	B_n=0\quad  ( \text{for $n\geq 3$ so that $2\nmid n$})
\end{align*}
Moreover, we know that the Bernoulli numbers are expressed as follows:
\begin{align}\label{bexp}
	B_n=(-1)^n\sum_{m=1}^{n+1}\frac{(-1)^{m-1}(m-1)! }{m}{n\brace m-1},
\end{align}
where ${n\brace m}\in\mathbb{Z}$ are the {\it Stirling numbers of the second kind} defined by
\begin{align}\label{defsti}
	\frac{(e^z-1)^{m} }{m!}=\sum^{\infty}_{n=0}{n\brace m}\frac{z^n}{n!}.
\end{align}

In 2014, K. Imatomi, M. Kaneko and E. Takeda \cite{IKT} concerned two types of the multi-poly-Bernoulli numbers which generalize the Bernoulli numbers.
\begin{defn}[\cite{IKT}, (1) and \cite{AK}, (8)]
\label{defmpbn}
For ${\mathfrak s}:=(s_1, \ldots, s_r)\in\mathbb{Z}^{r}$, the {\it multi-poly-Bernoulli numbers} (MPBNs for short) {\it of B-type, C-type} are the rational numbers which are defined by following generating functions respectively 
\begin{align*}
\sum_{n=0}^{\infty}B_n^{\mathfrak s}\frac{z^n}{n!}&:=\frac{{\rm Li}_{\mathfrak s}(1-e^{-z}, \overbrace{1, \ldots, 1}^{r-1})}{1-e^{-z}},\\
\sum_{n=0}^{\infty}C_n^{\mathfrak s}\frac{z^n}{n!}&:=\frac{{\rm Li}_{\mathfrak s}(1-e^{-z}, \overbrace{1, \ldots, 1}^{r-1})}{e^{z}-1}.
\end{align*}

\end{defn}
\begin{rem}
When $r=1$, $B_{n}^{s}$ and $C_{n}^{s}$ are the poly-Bernoulli numbers of B-type, C-type (cf. \cite{AK, MK}). When $r=1$ and $s_1=1$, ${\rm Li}_{\mathfrak s}(z_1, \ldots, z_r)=-\log(1-z)$ and then $B_n^{(1)}$ agrees with \eqref{defb} of the Bernoulli numbers. We note that $B_1^{(1)}=1/2$ and $C_1^{(1)}=-1/2$ and $B_n^{(1)}=C_n^{(1)}=B_n$ for $n\neq 1$. 
\end{rem}

\subsection{Characteristic $p$ case}
\label{2chp}
In 1935, L. Carlitz \cite{Ca} introduced the Bernoulli-Carlitz numbers, function field analogues of the Bernoulli numbers by using the Carlitz factorials $\Pi(n)$ and the {\it Carlitz exponentials} 
\[
e_C(z):=\sum_{i\geq0}\frac{z^{q^i} }{D_i}
\] 
as follows. 
\begin{defn}[\cite{Ca}]
\label{bc}
For $n\in\mathbb{Z}_{\geq0}$, the {\it Bernoulli-Carlitz numbers} $BC_n$ are the elements of $k$ defined by
\[	  
	 \sum_{n=0}^{\infty}BC_n\frac{z^n}{\Pi(n)}:=\frac{z}{e_{C}(z)}.
\]
\end{defn}
In \cite{Ca37}, L. Carlitz obtained the following: 
\[
	BC_n=0\quad \text{for $(q-1)\nmid n$}.
\]
In 2016, H. Kaneko and T. Komatsu \cite{KK} introduced the Stirling-Carlitz numbers of the first and second kind as an analogue of the Stirling numbers which were introduced in \eqref{defsti}. We recall below those of the second kind.

\begin{defn}[\cite{KK}, (15)]
\label{sc2}
For $m\in\mathbb{Z}_{\geq0}$, the {\it Stirling-Carlitz numbers of the second kind} ${n\brace m}_{C}\in k$ are defined by
\[
	\sum_{n=0}^{\infty}{n\brace m}_{C}\frac{z^n}{\Pi(n)}:=\frac{(e_C(z))^m}{\Pi(m)}.
\]
\end{defn}

In addition, they \cite{KK} showed that 
\begin{align}
\label{sc2pro}
{n\brace 0}_{C}=0\ (n\geq1),\quad {n\brace m}_{C}=0\ (n<m),\quad {n\brace n}_{C}=1\ (n\geq0)
\end{align}
and the following property.
\begin{prop}[\cite{KK}, Proposition 8]
\label{sc2van}
For $n,m\in\mathbb{Z}_{>0}$ with $\lambda(n)>\lambda(m)$,
\[
	{n\brace m}_C=0
\]
here we noted $\lambda(n):=\sum_{i}n_i$ where $n_i$ are the digits of $q$-adic expansion $n=\sum_{i}n_iq^i$.
\end{prop}

By using the Stirling-Carlitz numbers of the second kind, they obtained the following proposition as a function field analogue of \eqref{bexp}.
\begin{prop}[\cite{KK}, Theorem 2]
\label{kkr}
For $n\in\mathbb{Z}_{\geq0}$, we have
\[
	BC_n=\sum_{j=0}^{\infty}\frac{(-1)^jD_j}{L^2_{j}}{n\brace q^j-1}_{C}.
\]
\end{prop}
\begin{rem}
	In \cite{KK}, they put $L_j$ by $\prod^{j}_{i=1}(\theta^{q^i}-\theta)$. But the above equation is same to their equation (cf. \cite{KK}, (20)) due to the appearance of $L_j^2$. 
\end{rem}
\begin{rem}\label{factorialeq}
By the definition of $D_m$, we have
\[
	D_m^q=\prod^{m-1}_{i=0}(\theta^{q^{m}}-\theta^{q^{i}})^q=\prod^{m-1}_{i=0}(\theta^{q^{m+1}}-\theta^{q^{i+1}})=\prod^{m}_{i'=1}(\theta^{q^{m+1}}-\theta^{q^{i'}})=-\frac{D_{m+1}}{(\theta-\theta^{q^{m+1}})}.
\]
Thus we obtain 
\[
	D_m^{q-1}=-\frac{D_{m+1}}{D_m(\theta-\theta^{q^{m+1}})}.
\]
By the definition of Carlitz factorial, $L_j$ and the above equation, we have the following:
\begin{align}\label{factdl}
	\Pi(q^j-1)=\prod^{j-1}_{m=0}D_m^{q-1}=\prod^{j-1}_{m=0}-\frac{D_{m+1}}{D_m(\theta-\theta^{q^{m+1}})}= (-1)^j\frac{D_j}{L_j}\quad (j\in\mathbb{Z}_{\geq 0}).
\end{align}
Thus we may write the formula in Theorem \ref{kkr} as follows:
\[
	BC_n=\sum_{j=0}^{\infty}\frac{\Pi(q^j-1)}{L_{j}}{n\brace q^j-1}_{C}.
\]
\end{rem}
Next we introduce multi-poly-Bernoulli-Carlitz numbers (MPBCNs) as function field analogues of MPBNs (Definition \ref{defmpbn}).
It is defined by the following generating function.
\begin{defn}
\label{mpbc}
For ${\mathfrak s}=(s_1, \ldots, s_r)\in\mathbb{N}^r$, ${\bf j}=(j_1, \ldots, j_r)\in J_{\mathfrak s}$ (for $J_{\mathfrak{s}}$, see Notation \ref{notasym}), we define {\it multi-poly-Bernoulli-Carlitz numbers} (MPBCNs for short) $BC^{{\mathfrak s}, {\bf j}}_{n}$ to be elements of $k$ as follows:
\begin{align}\label{mpbcgen}
	\sum_{n\geq 0}BC^{{\mathfrak s}, {\bf j}}_{n}\frac{z^n}{\Pi(n)}:=\frac{Li_{\mathfrak s}(e_{C}(z)u_{1j_{1} }, u_{2j_{2} }, \ldots, u_{rj_{r} } ) }{e_{C}(z) }.
\end{align}
\end{defn}

\begin{rem}
\label{remMPBCN}
In the case when $r=1$ and $s_1=1$ in the above definition,
we have $Li_{\mathfrak s}(z_1, \ldots, z_r)=\log_{C}(z)$ and $J_{\mathfrak s}=\{ 0 \}, u_{1j_1}=u_{10}=1$ since $H_{ {s_1}-1 }=H_0=1$. Hence we recover the Definition \ref{bc} by
\[
	\sum_{n\geq 0}BC^{(1), (0)}_{n}\frac{z^n}{\Pi(n)}
	=\frac{\log_{C}(e_{C}(z)) }{e_{C}(z) }=\frac{z}{e_{C}(z)}.
\]
This is the one we have seen in Definition \ref{bc} so we have 
\begin{align}\label{bcfrommpbc}
	BC_{n}^{(1), (0)}=BC_n.
\end{align}	 
\end{rem}

\begin{rem}\label{mpbcnvan}
Let $g$ be a generator of $\mathbb{F}_{q}^{\times}$ then we have
\begin{align}\label{gen}
	g^n=1\Leftrightarrow (q-1)|n.
\end{align}

By the definition, it follows that $e_C(gz)=ge_C(z)$. Then by \eqref{mpbcgen} and Definition \ref{cmpl}, we have
\begin{align*}
	\sum_{n\geq 0}BC^{{\mathfrak s}, {\bf j}}_{n}\frac{(gz)^n}{\Pi(n)}&=\frac{Li_{\mathfrak s}(e_{C}(gz)u_{1j_{1} }, u_{2j_{2} }, \ldots, u_{rj_{r} } ) }{e_{C}(gz) }=\sum_{i_1>\cdots>i_r\geq0} e_C(gz)^{q^{i_1}-1} \frac{u_{1j_1}^{q^{i_1}}\cdots u_{rj_r}^{q^{i_r}} }{L_{i_1}^{s_1}\cdots L_{i_r}^{s_r}}\\
\intertext{by using $e_C(gz)=ge_C(z)$ and \eqref{gen},}
&=\sum_{i_1>\cdots>i_r\geq0} e_C(z)^{q^{i_1}-1} \frac{u_{1j_1}^{q^{i_1}}\cdots u_{rj_r}^{q^{i_r}} }{L_{i_1}^{s_1}\cdots L_{i_r}^{s_r}}=\frac{Li_{\mathfrak s}(e_{C}(z)u_{1j_{1} }, u_{2j_{2} }, \ldots, u_{rj_{r} } ) }{e_{C}(z) }\\
	&=\sum_{n\geq 0}BC^{{\mathfrak s}, {\bf j}}_{n}\frac{z^n}{\Pi(n)}.
\end{align*}
By comparing the coefficients of $z^n$, we have
$
	g^nBC^{{\mathfrak s}, {\bf j}}_{n}=BC^{{\mathfrak s}, {\bf j}}_{n}.  
$
Therefore we obtain the following by \eqref{gen}: 
\[
	BC^{{\mathfrak s}, {\bf j}}_{n}=0\quad \text{for $(q-1)\nmid n$}.
\]
\end{rem}
The MPBNs are defined for $s_i\in\mathbb{Z}$, on the other hand our MPBCNs are defined for $s_i\in\mathbb{N}$. It is because in Definition \ref{mpbc}, we use $u_{i{j_i}}$, the coefficients of $H_{s_i-1}$ which are defined for $s_i\in\mathbb{N}$. 
We remark that we do not have two kinds of MPBCNs as we do in Definition \ref{defmpbn}.

\section{Several properties of multi-poly-Bernoulli(-Carlitz) numbers}
In this section, we obtain function field analogues of some results in \cite{IKT}.
In subsection \ref{mainch0}, we recall their results in characteristic $0$ case. 
In subsection \ref{mainchp}, we prove their analogue in characteristic $p$ case.
\subsection{Characteristic 0 case}
\label{mainch0}
In \cite{IKT}, K. Imatomi, M. Kaneko, and E. Takeda obtained explicit formulae for MPBNs.
They are the following finite sums involving the Stirling numbers of the second kind.     
\begin{prop}[\cite{IKT}, Theorem 3]\label{iktexp}
For ${\mathfrak s}=(s_1, \ldots, s_r)\in\mathbb{Z}^r$ and $n\geq0$, we have
\[
	B^{\mathfrak s}_n=(-1)^n\sum_{n+1\geq m_1>m_2>\cdots>m_r>0}(-1)^{m_1-1}(m_1-1)! {n\brace m_1-1}\frac{ 1 }{m_1^{s_1}\cdots m_r^{s_r} }
\]
and
\[
	C^{\mathfrak s}_n=(-1)^n\sum_{n+1\geq m_1>m_2>\cdots>m_r>0}(-1)^{m_1-1}(m_1-1)! {n+1\brace m_1}\frac{ 1 }{m_1^{s_1}\cdots m_r^{s_r} }.
\]
\end{prop}

In \cite{IKT}, they derived the following relations between the MPBNs and the Bernoulli numbers for the special case $(s_1, \ldots, s_r)=(1, \ldots, 1)$.
\begin{prop}[\cite{IKT}, Proposition 4]\label{iktspind}
For $r\geq1$ and $n\geq r-1$, we have
\begin{align*}
	B_n^{(\overbrace{1, \ldots, 1}^{r})}&=\frac{1}{n+1}\binom{n+1}{r}B^{(1)}_{n-r+1},\\
	C_n^{(\overbrace{1, \ldots, 1}^{r})}&=\frac{1}{n+1}\binom{n+1}{r}C^{(1)}_{n-r+1}.
\end{align*}
\end{prop}

In \cite{IKT}, they obtained the following relations which connect the MPBNs and finite multiple zeta values. 

\begin{prop}[\cite{IKT}, Theorem 8]\label{IKTprop}
For ${\mathfrak s}=(s_1, \ldots, s_r)\in\mathbb{Z}^r$, we have
\begin{align*}
	\zeta_{\mathscr A}({\mathfrak s})_{(p)}&=-C^{(s_1-1, s_2, \ldots, s_r)}_{p-2}\ \text{\rm mod $p$}\\
	\intertext{and for $d\geq0$,}
	\zeta_{\mathscr A}(\underbrace{1, \ldots, 1}_{d}, s_1, \ldots, s_r)_{(p)}&=-C^{(s_1-1, s_2,  \ldots, s_r)}_{p-d-2}\ \text{\rm mod $p$}
\end{align*}
\end{prop}
Here we note that the second relation generalizes the first relation. 

\subsection{Characteristic $p$ case}
\label{mainchp}
We prove function field analogues of Proposition \ref{iktexp}-\ref{IKTprop}.
The following theorem is a function field analogue of Proposition \ref{iktexp}.
\begin{thm}
\label{mpbceq}
For $r\in\mathbb{N}$, ${\mathfrak s}=(s_1, \ldots, s_r)\in\mathbb{N}^r$, ${\bf j}=(j_1, \ldots, j_r)\in J_{\mathfrak s}$ and $n\in\mathbb{Z}_{\geq0}$,
\begin{align}\label{expl}
	BC^{{\mathfrak s}, {\bf j}}_{n}=\sum_{\log_{q}(n+1)\geq i_1>\cdots>i_r\geq0}\Pi(q^{i_1}-1){n\brace q^{i_1}-1}_{C}\frac{u_{1j_1}^{q^{i_1} }\cdots u_{rj_r}^{q^{i_r} } }{L_{i_1}^{s_1}\cdots L_{i_r}^{s_r}}.
\end{align}
\end{thm}
\begin{proof}
By Definition \ref{cmpl}, the right hand side of \eqref{mpbcgen} is translated as follows.
\begin{align*}
	\frac{Li_{\mathfrak s}(e_{C}(z)u_{1j_{1} }, u_{2j_{2} }, \ldots, u_{rj_{r} } ) }{e_{C}(z) }&=\sum_{i_1>\cdots>i_r\geq0} e_C(z)^{q^{i_1}-1} \frac{u_{1j_1}^{q^{i_1}}\cdots u_{rj_r}^{q^{i_r}} }{L_{i_1}^{s_1}\cdots L_{i_r}^{s_r}}
\end{align*}	
\begin{align*}	
	\intertext{by Definition \ref{sc2} for $m=q^{i_1}-1$,}
	&=\sum_{i_1>\cdots>i_r\geq0} \sum_{n\geq0}\Pi (q^{i_1}-1){n\brace q^{i_1}-1}_{C}\frac{z^n}{\Pi(n)} \frac{u_{1j_1}^{q^{i_1}}\cdots u_{rj_r}^{q^{i_r}} }{L_{i_1}^{s_1}\cdots L_{i_r}^{s_r}}\\
	&=\sum_{n\geq0}\sum_{i_1>\cdots>i_r\geq0}\Pi (q^{i_1}-1) {n\brace q^{i_1}-1}_{C} \frac{u_{1j_1}^{q^{i_1}}\cdots u_{rj_r}^{q^{i_r}} }{L_{i_1}^{s_1}\cdots L_{i_r}^{s_r}}\frac{z^n}{\Pi(n)}\\
	&=\sum_{n\geq0}\sum_{\log_{q}(n+1)\geq i_1>\cdots>i_r\geq0}\Pi (q^{i_1}-1) {n\brace q^{i_1}-1}_{C} \frac{u_{1j_1}^{q^{i_1}}\cdots u_{rj_r}^{q^{i_r}} }{L_{i_1}^{s_1}\cdots L_{i_r}^{s_r}}\frac{z^n}{\Pi(n)}.
\end{align*}   
Then by Definition \ref{mpbc}, we have
\[
	\sum_{n\geq 0}BC^{{\mathfrak s}, {\bf j}}_{n}\frac{z^n}{\Pi(n)}=\sum_{n\geq0}\sum_{\log_{q}(n+1)\geq i_1>\cdots>i_r\geq0}\Pi (q^{i_1}-1) {n\brace q^{i_1}-1}_{C} \frac{u_{1j_1}^{q^{i_1}}\cdots u_{rj_r}^{q^{i_r}} }{L_{i_1}^{s_1}\cdots L_{i_r}^{s_r}}\frac{z^n}{\Pi(n)}.
\]
By comparing the coefficients of $z^{n}$, we obtain the formula \eqref{expl}.
\end{proof}

\begin{rem}
When $r=1$ and $s_1=1$, we have $H_{s_1-1}=H_{0}=1$.
Then $J_{\mathfrak s}=\{ 0 \}, u_{1j_1}=u_{10}=1$ hence we have  
\begin{align*}
	BC_n^{(1), (0)}&=\sum_{\log_q(n+1)\geq i_1\geq0}\Pi(q^{i_1}-1){n \brace q^{i_1}-1}_{C}\frac{1}{L_{i_1}}\\
	\intertext{by using \eqref{factdl},}
	               &=\sum_{\log_q(n+1)\geq i_1\geq0}(-1)^{i_1}\frac{D_{i_1} }{L_{i_1}^2 }{n\brace q^{i_1}-1}_{C}.
\end{align*}

Therefore by Remark \ref{remMPBCN} our Theorem \ref{mpbceq} includes H. Kaneko and T. Komatsu's result (Proposition \ref{kkr}) in the case of $r=1$ and $s_1=1$. 
\end{rem}

We obtain the following relation between the MPBCNs and the Bernoulli-Carlitz numbers for the tuple $(1, \ldots, 1)$ as a function field analogue of Proposition \ref{iktspind}. 
\begin{thm}
For $r\geq1$ and $n\geq q^{r-1}-1$, we have
\begin{align}\label{mpbcsp}
	BC_n^{(\overbrace{1, \ldots, 1}^{r}), (\overbrace{0, \ldots, 0}^{r})}=\sum_{\log_q(n+1)\geq i_1>\cdots>i_r\geq0}{n \brace q^{i_1}-1}_C BC_{q^{i_1}-1}\frac{BC_{q^{i_2}-1}}{\Pi(q^{i_2}-1)}\cdots\frac{BC_{q^{i_r}-1}}{\Pi(q^{i_r}-1)}.
\end{align}
\end{thm}
\begin{proof}
Let us first prove an equation 
\begin{align}\label{osaka}
	\frac{BC_{q^{i}-1}}{\Pi(q^{i}-1)}=\frac{1}{L_{i}}.
\end{align}
It follows from Proposition \ref{kkr} that we have
\begin{align*}
	BC_{q^{i}-1}&=\sum_{j=0}^{\infty}\frac{(-1)^jD_j}{L_j^2}{q^i-1 \brace q^j-1}_C.
\end{align*}
The right hand side is translated as follows:
\[
	\sum_{j=0}^{\infty}\frac{(-1)^jD_j}{L_j^2}{q^i-1 \brace q^j-1}_C=\frac{(-1)^iD_i}{L_i^2}=\frac{\Pi(q^{i}-1)}{L_i} .
\]
The first equality follows from Proposition \ref{sc2van}, the second one follows from \eqref{factdl}.   
Then we have the equation \eqref{osaka}.

It follows from Theorem \ref{mpbceq} that we have
\[
	BC_n^{(\overbrace{1, \ldots, 1}^{r}), (\overbrace{0, \ldots, 0}^{r})}=\sum_{\log_q(n+1)\geq i_1>\cdots>i_r\geq0}{n \brace q^{i_1}-1}_C\frac{\Pi(q^{i_1}-1)}{L_{i_1}\cdots L_{i_r}}.
\]
By using the equation \eqref{osaka} to the right hand side,
\begin{align*}
	&BC_n^{(\overbrace{1, \ldots, 1}^{r}), (\overbrace{0, \ldots, 0}^{r})}\\
	&\quad =\sum_{\log_q(n+1)\geq i_1>\cdots>i_r\geq0}{n \brace q^{i_1}-1}_C\Pi(q^{i_1}-1) \frac{BC_{q^{i_1}-1}}{\Pi(q^{i_1}-1)}\frac{BC_{q^{i_2}-1}}{\Pi(q^{i_2}-1)}\cdots\frac{BC_{q^{i_r}-1}}{\Pi(q^{i_r}-1)}.
\end{align*}
Therefore we obtain the desired equation \eqref{mpbcsp}.
\end{proof}

Next, before we see a function field analogue of Proposition \ref{IKTprop}, we prepare the following lemma. 
\begin{lem} 
\label{mpbceq2}
When $r\geq 2$, we have the following equation for $\mathfrak{s}=(s_1, \ldots, s_r)\in\mathbb{N}^r$, ${\bf j}\in J_{\mathfrak s}$ and $m\geq r-1$.  
\begin{align}\label{lemeq}
	BC_{q^m-1}^{{\mathfrak s}, {\bf j}}=BC_{q^m-1}^{(s_1), (j_1)}\sum_{\alpha=1}^{m-(r-2)}\frac{1}{\Pi(q^{m-\alpha}-1)}BC_{q^{m-\alpha}-1}^{(s_2, \ldots, s_r), (j_2, \ldots, j_r)}.
\end{align}
\end{lem}
\begin{proof}
By using Theorem \ref{mpbceq}, we have
\begin{align*}
BC^{{\mathfrak s}, {\bf j}}_{q^m-1}=\sum_{m\geq i_1>\cdots>i_r\geq0}{q^m-1\brace q^{i_1}-1}_{C}\Pi(q^{i_1}-1)\frac{u_{1j_1}^{q^{i_1} }\cdots u_{rj_r}^{q^{i_r} } }{L_{i_1}^{s_1}\cdots L_{i_r}^{s_r}}.
\end{align*} 
All digits of the $q$-adic expansion of $q^{i_1}-1$ and $q^m-1$ are $q-1$. Therefore we have
\begin{align}\label{sc20or1}
	{q^m-1\brace q^{i_1}-1}_C=\begin{cases}
							0& \text{if $m>i_1$,}\\
							1& \text{if $m=i_1$,}
						\end{cases}
\end{align}
by Proposition \ref{sc2van} and \eqref{sc2pro}. Then by using \eqref{sc20or1}, we have
\begin{align*}
	BC^{{\mathfrak s}, {\bf j}}_{q^m-1}=\sum_{m>i_2>\cdots>i_r\geq0}\Pi(q^m-1)\frac{u_{1j_1}^{q^m }u_{2j_2}^{q^{i_2}}\cdots u_{rj_r}^{q^{i_r} } }{L_{m}^{s_1}L_{i_2}^{s_2}\cdots L_{i_r}^{s_r}}=\Pi(q^{m}-1)\frac{u_{1j_1}^{q^{m} }}{L_{m}^{s_1}}\sum_{m>i_2>\cdots>i_r\geq0}\frac{u_{2j_2}^{q^{i_2}}\cdots u_{rj_r}^{q^{i_r} } }{L_{i_2}^{s_2}\cdots L_{i_r}^{s_r}}.
\end{align*}
By using Theorem \ref{mpbceq}, we have
\begin{align}\label{totyu}
	BC^{{\mathfrak s}, {\bf j}}_{q^m-1}&=BC_{q^m-1}^{(s_1), (j_1)}\sum_{m>i_2>\cdots>i_r\geq0}\frac{u_{2j_2}^{q^{i_2}}\cdots u_{rj_r}^{q^{i_r} } }{L_{i_2}^{s_2}\cdots L_{i_r}^{s_r}}\\
	&=BC_{q^m-1}^{(s_1), (j_1)}\sum_{\alpha=1}^{m-(r-2)}\frac{1}{\Pi(q^{m-\alpha}-1) }\sum_{m-\alpha>i_3>\cdots>i_r\geq0}\Pi(q^{m-\alpha}-1)\frac{u_{2j_2}^{q^{m-\alpha} }u_{3j_3}^{q^{i_3}}\cdots u_{rj_r}^{q^{i_r} } }{L_{m-\alpha}^{s_2}L_{i_3}^{s_3}\cdots L_{i_r}^{s_r}}\nonumber\\
	\intertext{again by using Theorem \ref{mpbceq},}
	&=BC_{q^m-1}^{(s_1), (j_1)}\sum_{\alpha=1}^{m-(r-2)}\frac{1}{\Pi(q^{m-\alpha}-1)}BC_{q^{m-\alpha}-1}^{(s_2, \ldots, s_r), (j_2, \ldots, j_r)}.\nonumber
\end{align}
Then we obtain the desired equation \eqref{lemeq}.
\end{proof}

The following result is an analogue of Proposition \ref{IKTprop} which provides the connection between MPBCNs and finite multiple zeta values in the function field case. 

\begin{thm}\label{thmfin}
For ${\mathfrak s}=(s_1, \ldots, s_r)\in\mathbb{N}^r$ and a monic irreducible polynomial $\wp\in A$ so that $\wp \nmid \Gamma_{s_1}\cdots\Gamma_{s_r}$, we have the following:
\begin{align}\label{thmeq1}
	\zeta_{\mathscr{A}_{k}}({\mathfrak s})_{\wp}=\frac{1}{\Gamma_{s_1}\cdots\Gamma_{s_r}}\sum_{{\bf j}\in J_{\mathfrak s}}a_{\bf j}(\theta)\sum_{i=r-1}^{\deg\wp -1}\frac{1}{L_i}\frac{BC_{q^{i}-1}^{{\mathfrak s}, {\bf j} } }{BC_{q^i-1}}\mod \wp.
\end{align}

For ${\mathfrak s}=(\underbrace{1, \ldots, 1}_{d}, s_1, \ldots, s_r)\in\mathbb{N}^{r+d}$ $(d\geq 0)$ and a monic irreducible polynomial $\wp\in A$ so that $\wp \nmid \Gamma_{s_1}\cdots\Gamma_{s_r}$, we have the following:
\begin{align}\label{thmeq2}
\zeta_{\mathscr{A}_{k}}({\mathfrak s})_{\wp}&=\frac{1}{\Gamma_{s_1}\cdots\Gamma_{s_r}}\sum_{{\bf j'}\in J_{\mathfrak s'}}a_{\bf j'}(\theta)\sum_{\deg\wp>i_0>\cdots>i_d\geq r-1}\frac{1}{L_{i_0}\cdots L_{i_d}}\frac{BC^{{\mathfrak s'}, {\bf j'}}_{q^{i_d}-1}}{BC_{q^{i_d}-1}} \mod\wp.
\end{align}
Here we put ${\mathfrak s}'=(s_1, \ldots, s_r)$.
\end{thm}
\begin{proof}
We first prove that the equation \eqref{thmeq1}.
By \eqref{bcfrommpbc} and \eqref{totyu}, we have the following for $(1, s_1, \ldots, s_r)\in\mathbb{N}^{r+1}$:
\begin{align}\label{hokkaido}
	BC_{q^{\deg\wp}-1}^{(1, s_1, \ldots, s_r), (0, j_1, \ldots, j_r)}&=BC_{q^{\deg\wp}-1}^{(1), (0)}\sum_{\deg\wp>i_1>\cdots>i_r\geq0}\frac{u_{1j_1}^{q^{i_1}}\cdots u_{rj_r}^{q^{i_r} } }{L_{i_1}^{s_1}\cdots L_{i_r}^{s_r}}\nonumber\\ 
		&=BC_{q^{\deg\wp}-1}\sum_{\deg\wp>i_1>\cdots>i_r\geq0}\frac{u_{1j_1}^{q^{i_1}}\cdots u_{rj_r}^{q^{i_r} } }{L_{i_1}^{s_1}\cdots L_{i_r}^{s_r}}.
\end{align}
By Proposition \ref{sc2van} and Proposition \ref{kkr}, we have
\[
	BC_{q^{\deg\wp}-1}=(-1)^{\deg\wp}\frac{D_{\deg\wp}}{L_{\deg\wp}^2}.
\] 
Then $BC_{q^{\deg\wp}-1}$ is invertible in $k$ because of $D_{\deg\wp}, L_{\deg\wp}\in A_+$ and therefore by using \eqref{hokkaido}, we have
\begin{align*}
	\frac{BC_{q^{\deg\wp}-1}^{(1, s_1, \ldots, s_r), (0, j_1, \ldots, j_r)}}{BC_{q^{\deg\wp}-1}}=\sum_{\deg\wp>i_1>\cdots>i_r\geq0}\frac{u_{1j_1}^{q^{i_1}}\cdots u_{rj_r}^{q^{i_r} } }{L_{i_1}^{s_1}\cdots L_{i_r}^{s_r}}.
\end{align*}
By using equation \eqref{lemeq} in Lemma \ref{mpbceq2} for ${\mathfrak s}=(1, s_1, \ldots, s_r)$ and $m=\deg\wp$, we have
\begin{align*}
	&BC_{q^{\deg\wp}-1}^{(1, s_1, \ldots, s_r), (0, j_1, \ldots, j_r)}\\
	&\quad =BC_{q^{\deg\wp}-1}^{(1), (0)}\sum_{\alpha=1}^{\deg\wp-(r-1)}\frac{1}{\Pi(q^{\deg\wp-\alpha}-1)}BC_{q^{\deg\wp-\alpha}-1}^{(s_1, \ldots, s_r), (j_1, \ldots, j_r)}.\nonumber
\end{align*}
By equation \eqref{bcfrommpbc}, we have
\begin{align}\label{akita}
	\frac{BC_{q^{\deg\wp}-1}^{(1, s_1, \ldots, s_r), (0, j_1, \ldots, j_r)}}{BC_{q^{\deg\wp}-1}}&=\sum_{\alpha=1}^{\deg\wp-(r-1)}\frac{1}{\Pi(q^{\deg\wp-\alpha}-1)}BC_{q^{\deg\wp-\alpha}-1}^{(s_1, \ldots, s_r), (j_1, \ldots, j_r)}\nonumber\\
	\intertext{by putting $i=\deg\wp-\alpha$ and using the equation \eqref{osaka},}
		&=\sum_{i=r-1}^{\deg\wp -1}\frac{1}{L_i}\frac{BC_{q^{i}-1}^{(s_1, \ldots, s_r), (j_1, \ldots, j_r)}}{BC_{q^i-1}}
\end{align}
Therefore by the equations \eqref{hokkaido}, \eqref{akita} and Definition \ref{fcmpldef}, we obtain
\[
	\sum_{i=r-1}^{\deg\wp -1}\frac{1}{L_i}\frac{BC_{q^{i}-1}^{{\mathfrak s}, {\bf j} } }{BC_{q^i-1}}=Li_{\mathscr{A}_{k}, \mathfrak{s}}({\bf u_{j}})_{\wp}\ \text{mod $\wp$}.
\]
By our assumption $\wp\nmid \Gamma_{s_1}\cdots\Gamma_{s_r}$ we may apply Proposition \ref{fmzvfmpl} and obtain the desired formula \eqref{thmeq1}.

Next we prove the equation \eqref{thmeq2}. By using \eqref{thmeq1} for ${\mathfrak s}=(\underbrace{1, \ldots, 1}_{d}, s_1, \ldots, s_r)$, we have
\begin{align}\label{hyogo}
	\frac{1}{\Gamma_{1}^d\Gamma_{s_1}\cdots\Gamma_{s_r}}\sum_{{\bf j}\in J_{\mathfrak s}}a_{\bf j}(\theta)\sum_{\alpha=1}^{\deg\wp-(d+r-1)}\frac{BC^{{\mathfrak s}, {\bf j}}_{q^{\deg\wp-\alpha}-1}}{\Pi(q^{\deg\wp-\alpha}-1)}=\zeta_{\mathscr{A}_{k}}({\mathfrak s})_{\wp}\mod \wp
\end{align}
We may rewrite $BC^{{\mathfrak s}, {\bf j}}_{q^{\deg\wp-\alpha}-1}/\Pi(q^{\deg\wp-\alpha}-1)$ by using MPBCNs for $(s_1, \ldots, s_r)$. By \eqref{lemeq} in Lemma  \ref{mpbceq2}, 
\begin{align*}
	\frac{BC^{{\mathfrak s}, {\bf j}}_{q^{\deg\wp-\alpha}-1}}{\Pi(q^{\deg\wp-\alpha}-1)}
	&=\frac{BC^{(\overbrace{1, \ldots, 1}^{d}, s_1, \ldots, s_r), (\overbrace{0, \ldots, 0}^{d}, j_1, \ldots, j_r)}_{q^{\deg\wp-\alpha}-1}}{\Pi(q^{\deg\wp-\alpha}-1)}\\
	&=\frac{BC^{(1), (0)}_{q^{\deg\wp-\alpha}-1}}{\prod(q^{\deg\wp-\alpha}-1)}\sum^{\deg\wp-\alpha-(d+r-2)}_{\alpha_1=1}\frac{BC^{(\overbrace{1, \ldots, 1}^{d-1}, s_1, \ldots, s_r), (\overbrace{0, \ldots, 0}^{d-1}, j_1, \ldots, j_r)}_{q^{\deg\wp-\alpha-\alpha_1}-1}}{\Pi(q^{\deg\wp-\alpha-\alpha_1}-1)}
	\end{align*}
	\begin{align*}
	\intertext{by using \eqref{lemeq} repeatedly, }
	&=\frac{BC^{(1), (0)}_{q^{\deg\wp-\alpha}-1}}{\prod(q^{\deg\wp-\alpha}-1)}\sum^{\deg\wp-\alpha-(d+r-2)}_{\alpha_1=1}\frac{BC^{(1), (0)}_{q^{\deg\wp-\alpha-\alpha_1}-1}}{\prod(q^{\deg\wp-\alpha-\alpha_1}-1)}\sum^{\deg\wp-\alpha-\alpha_1-(d+r-3)}_{\alpha_2=1}\frac{BC^{(1), (0)}_{q^{\deg\wp-\alpha-\alpha_1-\alpha_2}-1}}{\prod(q^{\deg\wp-\alpha-\alpha_1-\alpha_2}-1)}\\
	&\quad\cdots\sum^{\deg\wp-\alpha-\alpha_1-\cdots-\alpha_{d-1}-(r-1)}_{\alpha_d=1}\frac{BC^{(s_1, \ldots, s_r), (j_1, \ldots, j_r)}_{q^{\deg\wp-\alpha-\alpha_1-\cdots-\alpha_{d}}-1}}{\prod(q^{\deg\wp-\alpha-\alpha_1-\cdots-\alpha_{d}}-1)}\\
	\intertext{by putting $\beta_i=\alpha+\alpha_1+\cdots+\alpha_i$ ($i\geq 1$) and $\beta_0=\alpha$,} 
	&=\frac{BC^{(1), (0)}_{q^{\deg\wp-\beta_0}-1}}{\prod(q^{\deg\wp-\beta_0}-1)}\sum^{\deg\wp-\beta_0-(d+r-2)}_{\beta_1-\beta_0=1}\frac{BC^{(1), (0)}_{q^{\deg\wp-\beta_1}-1}}{\prod(q^{\deg\wp-\beta_1}-1)}\sum^{\deg\wp-\beta_1-(d+r-3)}_{\beta_2-\beta_1=1}\frac{BC^{(1), (0)}_{q^{\deg\wp-\beta_2}-1}}{\prod(q^{\deg\wp-\beta_2}-1)}\\
	&\quad\cdots\sum^{\deg\wp-\beta_{d-1}-(r-1)}_{\beta_d-\beta_{d-1}=1}\frac{BC^{(s_1, \ldots, s_r), (j_1, \ldots, j_r)}_{q^{\deg\wp-\beta_{d}}-1}}{\prod(q^{\deg\wp-\beta_{d}}-1)}\\
	&=\frac{BC^{(1), (0)}_{q^{\deg\wp-\beta_0}-1}}{\prod(q^{\deg\wp-\beta_0}-1)}\sum^{\deg\wp-(d+r-2)}_{\beta_1=\beta_0+1}\frac{BC^{(1), (0)}_{q^{\deg\wp-\beta_1}-1}}{\prod(q^{\deg\wp-\beta_1}-1)}\sum^{\deg\wp-(d+r-3)}_{\beta_2=\beta_1+1}\frac{BC^{(1), (0)}_{q^{\deg\wp-\beta_2}-1}}{\prod(q^{\deg\wp-\beta_2}-1)}\\
	&\quad\cdots\sum^{\deg\wp-(r-1)}_{\beta_d=\beta_{d-1}+1}\frac{BC^{(s_1, \ldots, s_r), (j_1, \ldots, j_r)}_{q^{\deg\wp-\beta_{d}}-1}}{\prod(q^{\deg\wp-\beta_{d}}-1)}\\
	\intertext{by using \eqref{bcfrommpbc} and \eqref{osaka}, }
	&=\frac{1}{L_{\deg\wp-\beta_0}}\sum^{\deg\wp-(d+r-2)}_{\beta_1=\beta_0+1}\frac{1}{L_{\deg\wp-\beta_1}}\sum^{\deg\wp-(d+r-3)}_{\beta_2=\beta_1+1}\frac{1}{L_{\deg\wp-\beta_2}}\cdots\sum^{\deg\wp-(r-1)}_{\beta_d=\beta_{d-1}+1}\frac{BC^{(s_1, \ldots, s_r), (j_1, \ldots, j_r)}_{q^{\deg\wp-\beta_{d}}-1}}{\prod(q^{\deg\wp-\beta_{d}}-1)}.
\end{align*}
Then we have
\begin{align*}
\sum_{\beta_0=1}^{\deg\wp-(d+r-1)}&\frac{BC^{{\mathfrak s}, {\bf j}}_{q^{\deg\wp-\beta_0}-1}}{\Pi(q^{\deg\wp-\beta_0}-1)}=\sum_{\beta_0=1}^{\deg\wp-(d+r-1)}\sum^{\deg\wp-(d+r-2)}_{\beta_1=\beta_0+1}\cdots\sum^{\deg\wp-(r-1)}_{\beta_d=\beta_{d-1}+1}\frac{1}{L_{\deg\wp-\beta_0}}\frac{1}{L_{\deg\wp-\beta_1}}\\
	&\quad\cdots\frac{1}{L_{\deg\wp-\beta_{d-1}}}\frac{BC^{(s_1, \ldots, s_r), (j_1, \ldots, j_r)}_{q^{\deg\wp-\beta_{d}}-1}}{\prod(q^{\deg\wp-\beta_{d}}-1)}\\
	&=\sum_{\substack{\beta_d>\cdots>\beta_0>0\\ \deg\wp-(d+r-l-1)\geq\beta_l\text{for each $l$} }}\frac{1}{L_{\deg\wp-\beta_0}}\cdots\frac{1}{L_{\deg\wp-\beta_{d-1}}}\frac{BC^{(s_1, \ldots, s_r), (j_1, \ldots, j_r)}_{q^{\deg\wp-\beta_{d}}-1}}{\prod(q^{\deg\wp-\beta_{d}}-1)}\\
	&=\sum_{\deg\wp-(r-1)\geq\beta_d>\cdots>\beta_0>0}\frac{1}{L_{\deg\wp-\beta_0}}\cdots\frac{1}{L_{\deg\wp-\beta_{d-1}}}\frac{BC^{(s_1, \ldots, s_r), (j_1, \ldots, j_r)}_{q^{\deg\wp-\beta_{d}}-1}}{\prod(q^{\deg\wp-\beta_{d}}-1)}\\
\intertext{by putting $i_l=\deg\wp-\beta_l\ \ (d\geq l\geq 0)$,}
&=\sum_{\deg\wp>i_0>\cdots>i_d\geq r-1}\frac{1}{L_{i_0}}\cdots\frac{1}{L_{i_{d-1}}}\frac{BC^{(s_1, \ldots, s_r), (j_1, \ldots, j_r)}_{q^{i_d}-1}}{\prod(q^{i_d}-1)}
\intertext{by using the equation \eqref{osaka},}
&=\sum_{\deg\wp>i_0>\cdots>i_d\geq r-1}\frac{1}{L_{i_0}}\cdots\frac{1}{L_{i_{d}}}\frac{BC^{(s_1, \ldots, s_r), (j_1, \ldots, j_r)}_{q^{i_d}-1}}{BC_{q^{i_d}-1}}
\end{align*}
Substituting this into the equation \eqref{hyogo} and by $\Gamma_{1}=1$, we have 
\begin{align*}
\frac{1}{\Gamma_{s_1}\cdots\Gamma_{s_r}}\sum_{{\bf j}\in J_{\mathfrak s}}a_{\bf j}(\theta)&\sum_{\deg\wp>i_0>\cdots>i_d\geq r-1}\frac{1}{L_{i_0}\cdots L_{i_d}}\frac{BC^{(s_1, \ldots, s_r), (j_1, \ldots, j_r)}_{q^{i_d}-1}}{BC_{q^{i_d}-1}}
	=\zeta_{\mathscr{A}_{k}}({\mathfrak s})_{\wp}\ \text{mod $\wp$}.
\end{align*}

For ${\mathfrak s}=(1, \ldots, 1, s_1, \ldots, s_r)$, we have $J_{\mathfrak s}=\{0\}\times\cdots\times\{0\}\times\{0, 1, \ldots, m_{1} \}\times\cdots\times\{0, 1, \ldots, m_{r}\}$ so $a_{\bf j}(\theta)=\theta^{j_1+\cdots +j_r}$ for ${\bf j}=(0, \ldots, 0, j_1, \ldots, j_r)\in J_{\mathfrak s}$ and thus $a_{\bf j}(\theta)$ depends only on ${\bf j'}=(j_1, \ldots j_r) \in J_{\mathfrak s'}$. Therefore the above equation is rewritten as follows:
\begin{align*}
\frac{1}{\Gamma_{s_1}\cdots\Gamma_{s_r}}\sum_{{\bf j'}\in J_{\mathfrak s'}}a_{\bf j'}(\theta)&\sum_{\deg\wp>i_0>\cdots>i_d\geq r-1}\frac{1}{L_{i_0}\cdots L_{i_d}}\frac{BC^{{\mathfrak s'}, {\bf j'}}_{q^{i_d}-1}}{BC_{q^{i_d}-1}}
	=\zeta_{\mathscr{A}_{k}}({\mathfrak s})_{\wp}\ \text{mod $\wp$}.
\end{align*}
Thus we obtain the equation \eqref{thmeq2}.
\end{proof}
We remark that the relation \eqref{thmeq2} is a generalization of $\eqref{thmeq1}$.

\section*{Acknowledgments}
The author is deeply grateful to Professor H. Furusho for guiding him towards this topic. This paper could not have been written without his continuous encouragements. He gratefully acknowledges Professor Y. Mishiba for indicating him towards generalizations of Remark \ref{mpbcnvan} and the relation \eqref{thmeq2} which improved this paper. He would also like to thank NCTS for their kind support during his stay at NTHU and Daiko Foundation for financial support. 

\end{document}